\documentclass[12pt]{amsart}
\usepackage{amsmath,amstext,amsfonts,amsthm,amssymb,hyperref}
\usepackage{eucal}
\usepackage[all]{xy}                                     %
\swapnumbers

\newtheorem{lem}[subsubsection]{Lemma}

\newtheorem{THM}[subsection]{Theorem}

\newtheorem*{Thm}{Theorem}

{  \theoremstyle{definition}

           \newtheorem{rem}[subsubsection]{Remark}

           \newtheorem*{Dfn}{Definition}

}

\newcommand{\Cat}{\mathtt{Cat}}

\newcommand{\colim}{\operatorname{colim}}

\newcommand{\Fun}{\operatorname{Fun}}

\newcommand{\Hom}{\mathrm{Hom}} 

\newcommand{\id}{\mathrm{id}}

\newcommand{\one}{\mathbf{1}}
\newcommand{\Ob}{\operatorname{Ob}}
\newcommand{\op}{\mathrm{op}}

\newcommand{\Alg}{\mathtt{Alg}}

\newcommand{\cA}{\mathcal{A}}
\newcommand{\cB}{\mathcal{B}}

\newcommand{\cM}{\mathcal{M}}

\begin{document}

\title[]{Enriched Yoneda lemma}
\author{Vladimir Hinich}
\address{Department of Mathematics, University of Haifa,
Mount Carmel, Haifa 3498838,  Israel}
\email{hinich@math.haifa.ac.il}
 
\begin{abstract}We present a version of enriched Yoneda lemma for 
conventional (not $\infty$-) categories. We do not require the base monoidal category
$\cM$ to be closed or symmetric monoidal. In the case $\cM$ has colimits and the monoidal structure in $\cM$ preserves colimits in each argument, we prove that 
the Yoneda embedding $\cA\to P_\cM(\cA)$ is a universal functor from $\cA$ to a category
with colimits, left-tensored over $\cM$.   
\end{abstract}
\maketitle

\section{Introduction}

\subsection{}
The principal source on enriched category theory is the classical Max~Kelly's book ~\cite{K}. The theory is mostly developed under the assumption that the basic monoidal category $\cM$ is symmetric monoidal, and is closed, that is admits an internal Hom --- a functor right adjoint to the tensor product. 

The aim of this note is to present an approach which would make both conditions unnecessary.

Throughout the paper we study categories enriched over an arbitrary  monoidal category 
$\cM$. Note that this means that,
if $\cA$ is enriched over $\cM$, the opposite category $\cA^\op$ is enriched
over the monoidal category $\cM_\op$ having the opposite multiplication. 
Also, since we do not require $\cM$ to be closed, $\cM$ may not be enriched over itself.

Our approach is based on the following observation. Even though categories 
left-tensored over $\cM$ are not necessarily enriched over $\cM$, it  
makes a perfect sense to talk about $\cM$-functors $\cA\to \cB$ 
where $\cA$ is
$\cM$-enriched, and $\cB$ is left-tensored over $\cM$. 
Thus, $\cM$-enriched categories and categories left-tensored over $\cM$ appear in our approach as distinct but interconnected species.

\subsection{}
In this note we present two results in the enriched setting.
The first is construction of the category of enriched presheaves
and the Yoneda lemma. The second result, claiming a universal property of the 
category of enriched presheaves, requires $\cM$ to have colimits, so that the tensor product in $\cM$ preserves colimits in both arguments.

\subsection{}
In this note we adopt the language which allows us not to mention
associativity constraints explicitly. Thus is done as follows.
The small categories are considered  belonging to $(2,1)$-category 
$\Cat$, with functors as 1-morphisms and isomorphisms of functors as 
2-morphisms. Associative algebras in 2-category $\Cat$ 
are precisely monoidal categories, and left modules over these algebras
are left-tensored categories.
 
Similarly, we denote $\Cat^L$  the $(2,1)$-category whose objects are
the categories with colimits, 1-morphisms are  colimit preserving 
functors, and 2-morphisms are isomorphisms of such functors. 

This is a symmetric monoidal $(2,1)$-category, with tensor product 
defined by the formula
\begin{equation}
\Fun(A\otimes B,C)=\{f:A\times B\to C| f\textrm{ preserves colimits in 
both arguments}\}.
\end{equation} 
Associative algebras in $\Cat^L$ are monoidal categories with colimits, 
such that
tensor product preserves colimits in each argument
\footnote{As it is shown in \cite{L.HA}, Chapter 2, there is no 
necessity of keeping explicit track of various coherences even in the 
more general context of quasicategories.}.

\subsection{}
As it was pointed to us by the referee, enriched Yoneda lemma in the 
generality
presented in this note is not a new result. A recent paper \cite{GS} 
contains it (see 
Sections 5,7), as well as many other results, in even more general 
context of monoidal 
bicategories.
The approach of {\sl op. cit} is close to ours. The authors do not have the notion
of $\cM$-functor $\cA\to\cB$ from $\cM$-enriched category $\cA$ to a category $\cB$ 
left-tensored over $\cM$; but they construct the category of $\cM$-presheaves 
$P_\cM(\cA)$ ad hoc using the same formulas.  

We are very grateful to the referee for providing this reference, as well as for indicating that we do not use cocompletness of $\cM$ in Sections 2, 3.

\subsection{} The approach to Yoneda lemma presented in this note is very instrumental
in the theory of enriched infinity categories. We intend to address this in a subsequent publication.

\section{Two types of enrichment}

Let $\cM$ be a monoidal category.
In this section we define $\cM$-categories and categories
left-tensored over $\cM$.

\subsection{$\cM$-enriched categories}
Let $\cM$ be a monoidal category. An $\cM$-enriched category $\cA$
(or just $\cM$-category) has a set of objects, an object $\hom_\cA(x,y)\in\cM$
for each pair of objects (``internal Hom''), identity maps
$\one\to\hom(x,x)$ for each $x$ and associative compositions
$$\hom(y,z)\otimes\hom(x,y)\to\hom(x,z).$$

Let $\cA$ be $\cM$-enriched category. Its opposite $\cA^\op$ is a category
enriched over $\cM_\op$. The latter is the same category as $\cM$, but having the
opposite tensor product structure. The category $\cA^\op$ has the same objects as
$\cA$. Morphisms are defined by the formula
$$
\hom_{\cA^\op}(x^\op,y^\op)=\hom_\cA(y,x),
$$
with the composition defined in the obvious way.

\subsection{Left-tensored categories}

A left-tensored category $\cA$ over $\cM$ is just a left (unital) module 
for the associative algebra $\cM\in\Alg(\Cat)$. Note that unitality is 
not an extra structure, but a property saying that the unit
of $\cM$ acts on $\cA$ as an equivalence.

Right-tensored categories over $\cM$ are defined similarly. They are the same 
as the categories left-tensored over $\cM_\op$.

\begin{rem}
In case $\cM\in\Alg(\Cat^L)$, that is, $\cM$ has colimits and the 
monoidal operation in $\cM$ preserves colimits in each argument, we will 
define left-tensored categories over $\cM$ as left $\cM$-modules over 
the associative algebra $\cM\in\Alg(\Cat^L)$. 
A left-tensored category so defined has colimits, 
and the tensor product preserves colimits in both arguments.
\end{rem}

Left-tensored categories over $\cM$ often give rise to an $\cM$-enriched
structure: we can define $\hom(x,y)$ as an object of $\cM$ representing 
the functor
\begin{equation}\label{eq:internalhom}
m\mapsto\Hom(m\otimes x,y).
\end{equation}

Even if the above functor is not representable, we will use the notation
$\hom(x,y)$ to define the functor~(\ref{eq:internalhom}).

\

Note that left-tensored categories are categories (with extra 
structure).
Enriched categories are not, formally speaking, categories: maps from 
one object to another form an object of $\cM$ rather than a set.

\section{$\cM$-functors}

In this section we present two contexts for the definition of
a category of $\cM$-functors: from one category left-tensored over $\cM$
to another, and from an $\cM$-category to  a left-tensored category
over $\cM$.

\subsection{$\cA$ and $\cB$ are left-tensored}

Given two categories $\cA$ and $\cB$, left-tensored over $\cM$, one defines 
a category $\Fun_\cM(\cA,\cB)$ of $\cM$-functors as follows.

The objects are functors $f:\cA\to\cB$, together with 
a natural equivalence between two compositions in the diagram
\begin{equation}\label{eq:modmap}
\xymatrix{
&{\cM\otimes \cA}\ar[r]\ar[d]^{\id\otimes f} &{\cA}\ar[d]^f \\
&{\cM\otimes\cB}\ar[r]&{\cB}
}, 
\end{equation}
satisfying a compatibility in the diagram
\begin{equation}\label{eq:modmap2}
\xymatrix{
&{\cM\otimes\cM\otimes\cA}\ar@<1ex>[r]\ar@<-1ex>[r]\ar[d]^{\id\otimes\id\otimes f}
&{\cM\otimes \cA}\ar[r]\ar[d]^{\id\otimes f} 
&{\cA}\ar[d]^f \\
&{\cM\otimes\cM\otimes\cA}\ar@<1ex>[r]\ar@<-1ex>[r]
&{\cM\otimes\cB}\ar[r]&{\cB}
}  
\end{equation}
The morphisms in $\Fun_\cM(\cA,\cB)$ are morphisms of functors 
compatible with natural equivalences (\ref{eq:modmap}). Note that we 
have no unit condition on $f:\cA\to\cB$ as unitality of left-tensor 
categories is a property rather than extra data
\footnote{saying that the functor $\one\otimes:\cA\to\cA$ is an equivalence.}, so the ``unit constraints'' $\one\otimes x\to x$ are uniquely reconstructed and automatically preserved by $\cM$-functors.
 
In case $\cM\in\Alg(\Cat^L)$ and $\cA,\cB$ are left-tensored, we define 
$\Fun^L_\cM(\cA,\cB)$ as the category of colimit-preserving functors $f:
\cA\to\cB$, with a natural equivalence (\ref{eq:modmap}) satisfying 
compatibility (\ref{eq:modmap2}).

\subsection{$\cA$ is $\cM$-category and $\cB$ is left-tensored}

Let $\cA$ be $\cM$-enriched category and $\cB$ be left-tensored over $\cM$. We
will define $\Fun_\cM(\cA,\cB)$, the category of $\cM$-functors from $\cA$ to
$\cB$, as follows.

An $\cM$-functor $f:\cA\to\cB$ is given by a map 
$f:\Ob(\cA)\to\Ob(\cB)$, together with a compatible collection
of maps
\begin{equation}\label{eq:enr-to-mod}
\hom_\cA(x,y)\otimes f(x)\to f(y),
\end{equation}
given for each pair $x,y\in\Ob(\cA)$. 
The compatibility means that, given three objects $x,y,z\in\cA$,
one has a commutative diagram
\begin{equation}\label{eq:enr-to-mod2}
\xymatrix{
&{\hom_\cA(y,z)\otimes\hom_\cA(x,y)\otimes f(x)}\ar[r]\ar[d]
&{\hom_\cA(y,z)\otimes f(y)}\ar[d]\\ 
&{\hom_\cA(x,z)\otimes f(x)}\ar[r]
&{f(z).}
}
\end{equation}

Note that here, once more, we need no special unitality condition:
the map (\ref{eq:enr-to-mod}) applied to $x=y$ , composed with the 
unit $\one\to\hom_\cA(x,x)$, yields automatically the ``unit 
constraint'' $\one\otimes f(x)\to f(x)$: this follows from 
(\ref{eq:enr-to-mod2}) and the unitality of $\cB$.

$\cM$-functors from $\cA$ to $\cB$ form a category: a map from $f$ to
 $g$ is given by a compatible collection of arrows $f(x)\to g(x)$ in 
$\cB$ for any $x\in\Ob(\cA)$.

\subsection{$\cM$-presheaves}

The category $\cM$ is both left and right-tensored over $\cM$.
Given an $\cM$-category $\cA$, the opposite category $\cA^\op$ is
enriched over $\cM_\op$, so one has a category of $\cM_\op$-functors
$\Fun_{\cM_\op}(\cA^\op,\cM)$. We will call it {\sl the category of 
$\cM$-presheaves on $\cA$} and we will denote it $P_\cM(\cA)$.

\subsubsection{}
Let us describe explicitly what is an $\cM$-presheaf on $\cA$.
This is a map $f:\Ob(\cA)\to\Ob(\cM)$, together with a compatible
collection of maps
\begin{equation}\label{eq:Mpre}
f(y)\otimes\hom_\cA(x,y)\to f(x).
\end{equation}

\subsubsection{}
Let us show that $P_\cM(\cA)$ is left-tensored over $\cM$. Given 
$f\in P_\cM(\cA)=\Fun_{\cM_\op}(\cA^\op,\cM)$ and $m\in\cM$, the
 presheaf $m\otimes f$ is defined as follows.

It carries an object $x\in\cA^\op$ to $m\otimes f(x)$.
For a pair $x,y\in\Ob(\cA)$ the map
\begin{equation}
(m\otimes f(y))\otimes\hom_\cA(x,y)\to m\otimes f(x).
\end{equation}
is obtained from~(\ref{eq:Mpre}) by tensoring with $m$ on the left.

\subsubsection{}

The Yoneda embedding $Y:\cA\to P_\cM(\cA)$ is an $\cM$-functor
defined as follows.

For $z\in\cA$ the presheaf $Y(z)$ carries $x\in\cA$ to $\hom_\cA(x,z)\in\cM.$ 
The map (\ref{eq:Mpre})
\begin{equation}\label{eq:Y}
Y(z)(y)\otimes\hom_\cA(x,y)\to Y(z)(x)
\end{equation}
is defined by the composition 
$$\hom_\cA(y,z)\otimes\hom_\cA(x,y)\to\hom_\cA(x,z).$$

\begin{lem}\label{lem:Yoneda1}
The functor $\hom_{P_\cM(\cA)}(Y(x),F)$
is represented by $F(x)\in\cM$.
\end{lem}
\begin{proof}
The map of presheaves 
\begin{equation}\label{eq:FYF}
F(x)\otimes Y(x)\to F
\end{equation}
is given by the collection of maps
$F(x)\otimes\hom(z,x)\to F(z)$ which is a part of data for $F$.

We have to verify that (\ref{eq:FYF}) is universal. That is, any map
$\alpha:m\otimes Y(x)\to F$ in $P_\cM(\cA)$ comes from a unique map 
$\tilde\alpha:m\to F(x)$. The map $\tilde\alpha$ is the composition
$$ m\to m\otimes\hom_\cA(x,x)\to F(x).$$

\end{proof}

\subsubsection{Yoneda lemma}

Lemma~\ref{lem:Yoneda1} is a version of Yoneda lemma. 
Theorem~\ref{thm:yoneda} below saying Yoneda embedding is fully faithful is almost an immediate corollary.

\begin{Dfn}
An $\cM$-functor $f:\cA\to\cB$ from an $\cM$-category 
to an enriched category is {\sl fully faithful} if for 
any $x,y\in\cA$ the functor $\hom_\cB(f(x),f(y))$ defined by the 
formula~(\ref{eq:internalhom}), is represented by $\hom_\cA(x,y)$.
\end{Dfn}

\begin{Thm}\label{thm:yoneda}
The Yoneda embedding $Y:\cA\to P_\cM(\cA)$ is fully faithful for any small 
$\cM$-category $\cA$. 
\end{Thm}
\begin{proof}
Let $x,y\in\cA$. 
We have to prove that the canonical map 
$$\hom_\cA(x,y)\otimes Y(x)\to Y(y)$$
induces a bijection
\begin{equation}
\Hom_\cM(m,\hom_\cA(x,y))\to\Hom_{P_\cM(\cA)}(m\otimes Y(x),Y(y)).
\end{equation}
This is a special case of Lemma~\ref{lem:Yoneda1}.
\end{proof}

\section{Universal property of $\cM$-presheaves}

In this section we assume $\cM\in\Alg(\Cat^L)$.

The Yoneda embedding $Y:\cA\to P_\cM(\cA)$ induces, for each left-tensored category $\cB$ over $\cM$, a natural map
\begin{equation}\label{eq:univP}
Res:\Fun^L_\cM(P_\cM(\cA),\cB)\to\Fun_\cM(\cA,\cB).
\end{equation}

In this section we will show that the above map is an equivalence of categories.
In other words, we will prove that $P_\cM(\cA)$ is the universal left-tensored
category over $\cM$ with colimits generated by $\cA$.

\subsection{Weighted colimits}

Let, as usual, $\cA$ be $\cM$-category and $\cB$ be left-tensored over $\cM$.
Given $W\in P_\cM(\cA)$ and $F:\cA\to\cB$, we define the weighted colimit
$Z=\colim_W(F)$ as a object of $\cB$ together with a collection of arrows
$\alpha_x:W(x)\otimes F(x)\to Z$ making the diagrams
\begin{equation}
\xymatrix{
&{W(y)\otimes\hom_\cA(x,y)\otimes F(x)}\ar[d]\ar[r]
&{W(y)\otimes F(y)}\ar[d]^{\alpha_y}\\
&{W(x)\otimes F(x)}\ar[r]^{\alpha_x} &{Z}
}
\end{equation}
commutative for each pair $x,y\in\cA$, and satisfying an obvious universal 
property.
 
\

It is clear from the above definition that weighted colimits are special
kind of colimits, so they always exist.

Weighted colimit is a functor
$$
P_\cM(\cA)\times\Fun_\cM(\cA,\cB)\to \cB
$$
preserving colimits in both arguments.

Weighted colimits are very convenient in presenting presheaves as colimits of representable presheaves. This can be done in a very canonical way:
any presheaf $F\in P_\cM(\cA)$ is the weighted colimit
$$ F=\colim_F(Y),$$
where $Y:\cA\to P_\cM(\cA)$ is the Yoneda embedding.

\begin{THM}\label{thm:universal}
The functor (\ref{eq:univP}) is an equivalence of categories.
\end{THM}
\begin{proof}
We will construct a functor $Ext$ in the opposite direction. Given 
$F\in\Fun_\cM(\cA,\cB)$, we define $Ext(F)$ by the formula
\begin{equation}
Ext(F)(W)=\colim_W(F).
\end{equation}
It is easily verified that  the functors $Ext$ and $Res$ form
a pair of equivalences.

\end{proof}

\end{document}